\documentclass[a4paper]{article}
\usepackage{graphicx} 
\usepackage{amsmath, amsthm, amsfonts, amssymb, amscd, bm, enumerate}
\usepackage{verbatim}

\usepackage{url}

\newtheorem{theorem}{Theorem}
\newtheorem{corollary}{Corollary}
\newtheorem{lemma}{Lemma}
\newtheorem{definition}{Definition}
\newtheorem{assumption}{Assumption}

\theoremstyle{remark}
\newtheorem{remark}{Remark}

\def\F{\mathcal{F}}
\def\G{\mathcal{G}}
\def\E{\mathcal{E}}

\def\I{\mathcal{I}}

\def\bR{\mathbb{R}}
\def\bE{\mathbb{E}}
\def\bN{\mathbb{N}}

\def\bP{\mathbb{P}}
\def\bJ{\mathbb{J}}
\def\bS{\mathbb{S}}

\begin{document}
\title{Chaos representations for Marked Point Processes.}
\author{Samuel N. Cohen\\ samuel.cohen@maths.ox.ac.uk\\University of Oxford}

\date{\today}

\maketitle

\begin{abstract}
We show that for a large class of marked point processes there exists a random measure $m$ with the predictable representation property such that iterated integrals with respect to $m$ span the space of square integrable random variables.
\end{abstract}

\section{Introduction}
A fundamental result in the stochastic analysis on Wiener spaces is the Wiener-It\=o chaos representation theorem \cite{Wiener1938}. This theorem allows the representation of any square integrable random variable as the sum of iterated stochastic integrals with respect to the underlying Wiener process, and provides an approach to the Malliavin calculus of variations. Such a representation is often termed a chaos representation, and is closely linked to representations in terms of Hermite polynomials.

For processes with jumps, it is also possible to construct a theory of chaos expansions. This has been studied in the context of Markov chains in Kroeker  \cite{Kroeker1980} (see also Biane \cite{Biane1990}) and for the Binomial process in Privault and Schoutens \cite{Privault2002}.  In these works, the approach is based on the principle of finding an analogue to the Hermite polynomials appropriate to these spaces. Emery \cite{Emery1989} studies chaos representations in terms of iterated integrals, assuming that the underlying martingale satisfies a certain structure condition, related to the Az\'ema martingale. Many authors have since expanded on these ideas.

In this paper we give a general approach to chaos decompositions for an arbitrary marked point process, where we simply assume that the compensating measure for the underlying process is absolutely continuous (in both time and space) with respect to some (locally finite in time) deterministic measure. Instead of searching for a polynomial chaos interpretation, we focus on the representation in terms of iterated stochastic integrals with respect to fundamental martingales. For this, we use the fundamental martingales constructed in Elliott \cite{Elliott1976} and Davis \cite{Davis1976}, which make no assumptions about independence of the increments. In such a setting, we shall show that the iterated stochastic integrals span the entirety of $L^2(\F)$.

This result complements the construction of the Malliavin calculus for Marked Point Processes as in Decreusefond \cite[Section 4]{Decreusefond1998}. In \cite{Decreusefond1998} it is simply assumed that the space under consideration admits a chaos decomposition. The contribution of this paper is to show general conditions under which this is the case.

\section{Martingales for Marked Point Processes}
We begin by constructing an explicit martingale representation result. We do this mainly for copleteness, in Section \ref{sec:Chaos} we shall simply assume that some martingale representation is given, which may or may not be the one constructed here.

The setting for this analysis is taken from Elliott \cite{Elliott1976} (see also Davis \cite{Davis1976}, and Elliott \cite{Elliott1977}), we shall state relevant results without proof or further reference. For a simpler and more gentle introduction to this style of analysis for marked point processes, see \cite[Ch. 17]{Elliott1982}.

\subsection{The jump setting}\label{subsec:setting}
Let $(E, \E)$ be a Blackwell space. Consider a right-constant jump process $X$ taking values in $E$, which is initially in the fixed position $X_0 = \xi_0\in E$.

At a random time $T_1$, $X$ jumps to a random location $\xi_1\neq \xi_0$, at which it stays until a random time $T_2$, when it jumps to a random location $\xi_2\neq \xi_1$, etc... As $X$ is right-constant, we know that for each path, the jumps $T_i$ are well ordered, and there are at most countably many jumps. For simplicity, in this paper we shall assume that there are at most finitely many jumps on any compact, that is, $\lim_{n\to\infty} T_n = \infty$ for (almost) all paths.

We then have a measurable space $(\Omega, \F)$, where $\F = \sigma\{X_s, s<\infty\}$, and $\Omega \subset ([0,\infty]\times E)^\bN$ is a list of all the jump times and locations of $X$, with the property that $X$ can only jump once at each time, and must jump to a new location. We suppose a probability measure $\bP$ is given on this space.  We denote by $\F_t$ the $\bP$-completed $\sigma$-algebra generated by $X$ up to time $t$, that is $\F_t = \sigma\{X_s; s\leq t\}\vee\{\text{null sets}\}$. This space will be kept fixed throughout the paper.

\subsection{Fundamental martingales}
Suppose $T_\alpha$, $\alpha\in\bN$ is a jump time. The distribution of the pair $(T_{\alpha}, \xi_{\alpha})$ given $\F_{T_{\alpha-1}}$ is described by a random measure (that is, a regular family of conditional probability distributions) $\mu^{\alpha}(\omega; \cdot)$ on $[0,\infty]\times E$. Properties of $\mu^\alpha$ are given in \cite{Elliott1977}. Define
\[F^\alpha_t(\omega; A) = \mu^\alpha(\omega; ]t,\infty]\times A)\]
 so that, omitting $\omega$ for notational convenience,  $F_t^\alpha(A)$ is the conditional probability that $T_{\alpha} >t$ and $\xi_{\alpha} \in A$ given $\F_{T_{\alpha-1}}$. For convenience $F^\alpha_t:=F^\alpha_t(E)$ and we write
\[\lambda^\alpha(t, A)= \left.\frac{dF^\alpha_\cdot(A)}{dF^\alpha_\cdot(E)}\right|_t,\]
the rate at which the $\alpha$th jump is into $A$ at time $t$.  We can then define the stochastic processes
\[\begin{split}
   p^\alpha(t, A) &:= I_{t\geq T_{\alpha}} I_{\xi_{\alpha} \in A}\\
   \tilde p^\alpha(t, A) &:= -\int_{]0, t\wedge T_{\alpha}]}  (F_{u-}^{\alpha})^{-1} dF_u^{\alpha} (A) = -\int_{]0,t\wedge T_\alpha]}\lambda^\alpha(s,A) (F^\alpha_{s-})^{-1}dF_s^\alpha\\
   q^\alpha(t, A) &:= p^\alpha(t, A) - \tilde p^\alpha(t, A)
  \end{split}
\]
so that $q^\alpha(t, A)$ is an $\F_t$-martingale with predictable quadratic variation
 \[\langle q^\alpha(t, A)\rangle = \tilde p^\alpha(t, A) - \sum_{0< u\leq t\wedge T_{\alpha}} \frac{\lambda^\alpha(u, A)^2 (\Delta F^{\alpha}_u)^2}{(F^{\alpha}_{u-})^2}.\]
We shall see that these martingales provide a basis from which we can obtain a martingale representation theorem in these spaces. Note that $\tilde p^\alpha$ is simply the compensator of the finite variation process $p^\alpha$, and $q^\alpha$ is then the martingale part of $p^\alpha$. Note also that if $\tilde p^\alpha$ is continuous in $t$, then $\Lambda^\alpha(t, A)\Delta F^\alpha \equiv 0$, and $\langle q^\alpha(t, A)\rangle = \tilde p^\alpha(t,A)$.

Write $G^\alpha$ for the set of measurable functions $\{g^\alpha:\Omega\times[0,\infty]\times E\to \bR\}$ such that for each $(t,x)\in [0,T]\times E$, $g^\alpha$ is $\F_{T_{\alpha-1}}$-measurable. As for fixed $\alpha, t$ and $\omega$ we know $p^\alpha(t, A)$ and $\tilde p^\alpha(t,A)$ are both countably additive in $A$, for suitable $g^\alpha \in G^\alpha$ we have
\[\begin{split}
   \int_{\Omega} g^\alpha(s, x) p^\alpha(ds, dx) &= g^\alpha(T_{\alpha}, x_{\alpha})\\
   \int_{\Omega} g^\alpha(s, x) \tilde p^\alpha(ds, dx) &= - \int_{]0, T_{\alpha+1}]} \int_{E} g^\alpha(s,x) \lambda^\alpha(\omega; s, dx) \frac{ dF^\alpha_s}{F^\alpha_{s-}}
  \end{split}
\]

\begin{lemma}\label{lem:Msumofdeltas}
 For any square-integrable martingale $M$, we define
\[\Delta M^\alpha := M_{T_\alpha} - M_{T_{\alpha-1}}.\]
Then for every $\alpha\in \bN$, we have
\[M_{T_\alpha} = \sum_{\beta\leq \alpha} \Delta M^\beta.\]
\end{lemma}

This leads to a precursor to the martingale representation result in this context.
\begin{theorem}\label{thm:martdiffrepMPP}
Suppose $M$ is a square-integrable martingale; write
\[N^\alpha_t = M_{T_{\alpha}\wedge t} - M_{T_{\alpha-1} \wedge t}\]
Then for each $\alpha\in\bN$, there exists a function $g^\alpha\in G^\alpha$ such that
\[N^\alpha_t = \int_{]0,t]\times E} g^\alpha(s, x) q^\alpha(ds, dx) \quad a.s.\]
We shall say that $g^\alpha$ represents $N^\alpha$. Furthermore, $g^\alpha(T_\alpha, x_\alpha) = \Delta N^\alpha_{T_\alpha}$, up to the addition of a $\F_{T_{\alpha-1}}$-measurable random variable.
\end{theorem}

\section{Martingale representation theorem}
We now depart from the presentation of the martingale representation theorem in \cite{Elliott1977}, to present a slight variant which more naturally leads to the chaos representation, and is of a more familiar form. Our presentation depends on the following lemma and associated definition.

\begin{lemma}
For each $\omega$, any $t$, any $A\in\E$,  $p^\alpha(\omega;t, A)$ and $\tilde p^\alpha(\omega;t, A)$ vary in $t$ only on the set $t\in]T_{\alpha-1}, T_{\alpha}]$. In particular, the measures $\{dp^\alpha\}_{\alpha\in\bJ}$  on $[0,\infty]\times E$ have disjoint supports, and similarly for $\{d\tilde p^\alpha\}_{\alpha\in\bJ}$.

Therefore, we can define the disjoint sum
\[p(\omega; dt, dx) = \sum_{\alpha\in\bN} p^\alpha(\omega; dt, dx),\]
and similarly for $\tilde p$ and hence for $q=p-\tilde p$. 
\end{lemma}

\begin{corollary}\label{cor:martdiffrepMPPzeroed}
 Let $g^\alpha$ be as in Theorem \ref{thm:martdiffrepMPP}. Let $\tilde g^\alpha$ be defined as
\[\tilde g^\alpha(t,x) = I_{\{t\in]T_{\alpha-1}, T_\alpha]\}} g^\alpha(t,x)\]
Then $\tilde g^\alpha$ also represents $N^\alpha$
\end{corollary}
\begin{proof}
 This follows as we have only modified $g^\alpha$ off the support of $q^\alpha$.
\end{proof}

We can now state our first martingale representation theorem.
\begin{theorem}
 Let $M$ be a square-integrable $\{\F_t\}$-martingale. Then there exists an $\{\F_t\}$-predictable process $g(t, x)$ such that
\[M_t = M_0 + \int_{]0,t]\times E} g(t,x) q(dt, dx).\]
\end{theorem}
\begin{proof}
Let $\tilde g^\alpha$ be as in Corollary \ref{cor:martdiffrepMPPzeroed}. By Lemma \ref{lem:Msumofdeltas} and the fact $\tilde g^\alpha$ represents $N^\alpha$, we have
\[M_t-M_0 = \sum_{\alpha} N^\alpha_t  = \sum_{\alpha}\int_{]0,t]\times E} \tilde g^\alpha(s,x) q(ds, dx).\]
We then define $g(s, x) := \sum_{\alpha\leq\gamma} \tilde g^\alpha(s, x)$, this again being a disjoint sum. As there are almost surely finitely many jumps up to time $t$, and $\tilde g^\alpha$ is zero for $\alpha$ greater than the index of the next jump, for almost all $\omega$ this is a finite sum, and so we can exchange the order of integration and summation.
\end{proof}

This martingale representation theorem has a simple interpretation, as it is based purely on the compensated indicator functions of the state of the underlying process $X$. However, it has a significant flaw for our purposes, as iterated integrals are not necessarily orthogonal. For this reason, we need to rescale $q$, for which we need the following assumption. This assumption poses the only restriction on the processes we shall consider.

\begin{assumption}\label{Ass1}
 For all $\alpha$, there exists a \emph{deterministic} measure $\zeta^\alpha$ on $\bR^+\times E$ such that  $\tilde p^\alpha(\omega, \cdot, \cdot)$ is almost surely equivalent to $\zeta^\alpha$, and such that $\zeta^\alpha([0,t]\times E)<\infty$ for all $t<\infty$. For simplicity, we shall assume that $\zeta^\alpha$ is continuous with respect to $t$.
\end{assumption}

\begin{lemma}
 There exists a predictable function $\psi:\Omega \times \bR^+\times E \to ]0,1]$ such that for all measurable functions $f$, for all $\alpha\in \bJ$,
\[\begin{split}
   &\bE\left[\int_{]T_\alpha,T_{\alpha+1}]\times E} f(\omega, s,x) \psi(\omega, s,x) \tilde p(\omega, ds, dx)\right]\\
&= \int_{\bR^+\times E}\bE[I_{s\in]T_\alpha,T_{\alpha+1}]}f(\omega, s, x)] \zeta^\alpha(du,dx)
  \end{split}
\]
\end{lemma}
\begin{proof}
Simply take
\[\psi(\omega,t,x)=\sum_\alpha I_{t\in]T_{\alpha-1}, T_{\alpha}]}\left(\left.\frac{d\zeta^\alpha}{d\tilde p^\alpha(\omega, \cdot, \cdot)}\right|_{(t,x)}\right).\]
\end{proof}

\begin{definition}
We shall denote by $q_\psi$ the signed measure $q$ rescaled by $\psi$, that is,
\[q_\psi(t, A) := \int_{]0,t]\times A} \psi(\omega, s,x)q(\omega, ds,dx).\]
For simplicity, we may write $\psi_{t,x}$ for $\psi(\omega, t,x)$.
\end{definition}

\begin{lemma}
If $f$ is $q$-integrable, then $f \cdot \psi^{-1}$ is $q_\psi$-integrable and the two integrals agree. If $f$ is predictable, then so is $f \cdot \psi^{-1}$.
\end{lemma}
\begin{proof}
 This is clear as $\psi$ is predictable and for each $\omega$ equals the Radon-Nikodym derivative $dq_\psi/dq$.
\end{proof}

Using the previous lemma, we immediately see that our martingale representation theorem can be equivalently stated in terms of $q_\psi$, rather than $q$. This will be preferable, as $q_\psi$ has significantly better orthogonality properties than $q$, and so we shall hereafter focus on $q_\psi$.

We now seek to understand the space of integrands which yield square integrable martingales, when integrated with respect to $q_\psi$. As our martingale representation is not given by an orthonormal set of martingales, but rather by a random measure $q_\psi$ with $q_\psi(t,A)$ and $q_\psi(t, B)$ correlated, we need to be careful in our definition of the appropriate space of integrands.

\begin{lemma}
For all $f, g$ such that $\int_{]0,t]\times E} f(t,x)q_\psi(dt, dx)$ is square integrable (and similarly for $g$), we have the isometry
\[\begin{split}
 (f,g)_{q_\psi}&:=\bE\left[\left(\int_{\bR^+\times E} f(t,x) q_\psi(dt, dx)\right)\left(\int_{\bR^+\times E} g(t,x) q_\psi(dt, dx)\right)\right] \\
   &= \bE\left[\int_{\bR^+\times E}f(t,x)g(t, x) \psi_{t,x}^2 \tilde p(\omega, dt, dx)\right]\\
 &=\sum_\alpha \int_{\bR^+\times E}\bE\left[I_{t\in]T_\alpha,T_{\alpha+1}]}f(t,x)g(t, x)\right] \zeta^\alpha(dt, dx)
  \end{split}
\]
We shall write $\|f\|_{q_\psi}^2= (f,f)_{q_\psi}$.
\end{lemma}

\begin{proof}
From \cite{Elliott1977}, we know that the quadratic variation of $q$ is given by $\tilde p$, as we have assumed that $\zeta^\alpha$, and hence $\tilde p^\alpha$, is continuous in $t$. As $q_\psi$ is simply a rescaled version of $q$, this quickly establishes the first isometry. The second then follows by breaking up the integral into the intervals $]T_\alpha, T_{\alpha+1}]$, and extracting the sum.
\end{proof}

From this lemma, we can see that our use of the martingale random measure $q_\psi$ is a slight generalisation of constructing a martingale representation using `normal' martingales, that is, martingales with predictable quadratic variation given by Lebesgue measure (see, for example, Emery \cite{Emery1989}). Here we replace Lebesgue measure with an arbitrary deterministic measure $\zeta^\alpha$, which can vary in $\alpha$, and we retain the presence of the jump space $E$.

\section{Chaos representation property}\label{sec:Chaos}

From this point onwards, we will not restrict ourselves to this particular choice of martingale representation. In fact, there may be cases where an alternative martingale representation is available and more convenient. We shall simply make the following assumption.
\begin{assumption}
 We are in the setting described in Section \ref{subsec:setting}, and there exists a random measure $m$ such that
\begin{itemize}
 \item $\int_{]0,t]\times E} f(t,x) m(dt, dx)$ is a martingale for all predictable, sufficiently bounded functions $f$,
 \item every square integrable martingale has a representation $\int_{]0,t]\times E} f(t,x) m(dt, dx)$ for some predictable function $f$,
 \item for all sufficiently integrable predictable $f$ and $g$,
\[\begin{split}
 (f,g)_m&:=\bE\left[\left(\int_{\bR^+\times E} f(t,x) m(dt, dx)\right)\left(\int_{\bR^+\times E} g(t,x) m(dt, dx)\right)\right] \\
 &=\sum_\alpha \int_{\bR^+\times E}\bE\left[I_{t\in]T_\alpha,T_{\alpha+1}]}f(t,x)g(t, x)\right] \zeta^\alpha(dt, dx).
  \end{split}
\]
for some family of deterministic measures $\zeta^\alpha$. As before $\|f\|_m^2 := (f,f)_m$.
\end{itemize}
\end{assumption}

 Under Assumption \ref{Ass1}, $m=q_\psi$ satisfies these requirements. However, it may be convenient to take an alternative representation, particularly in cases when Assumption \ref{Ass1} does not hold. A simple example of this is when $E$ posesses a group structure (e.g. when $E$ is a vector space). If $E$ is discrete, for example, when we consider a countable-state Markov chain, then the representation based on $p$ will often not satisfy Assumption \ref{Ass1}, as the previous state $\xi_{\alpha-1}$ is a null set of the measure $p^\alpha$, however is stochastic, which often contradicts the equivalence with the deterministic measure $\zeta^\alpha$. On the other hand, we could use a representation based on the fundamental processes $\pi^\alpha(t, A) = I_{t\geq T_\alpha} I_{\xi_\alpha-\xi_{\alpha-1} \in A}$ (in the place of $p^\alpha$), that is, we use the indicator functions of the jumps themselves, rather than the indicator of the location after the jump. This representation (appropriately rescaled) will satisfy our assumption as soon as the set of possible values (occuring with rate $>0$) for the $\alpha$th jump is deterministic.

Using the martingale $m$, we now prove the existence of the Chaos representation of a random variable.

\begin{definition}
For two (stopping) times $T, T'\leq \infty$, we shall write
\[\F_{T\curlywedge T'} = \F_T\cap \F_{T'-}\]
and
\[\int_0^{T\curlywedge T'} (\cdot) m(dt, dx) := \int_{(]0, T]\cap ]0,T'[)\times E} (\cdot) m(dt, dx).\]
For simplicity, in place of $m(dt, dx)$ we may write $dm$, similarly $dm_1$ for $m(dt_1, dx_1)$, $dm_2$ for $m(dt_2, dx_2)$, etc. and also $d\zeta^\alpha$ for $\zeta^\alpha(dt,dx)$, $d\zeta^\alpha_1$ for $\zeta^\alpha(dt_1,dx_1)$, etc.
\end{definition}

Note that if $\tau=\infty$, then $L^2(\F_{T\curlywedge\tau})=L^2(\F_{T})$, as we have assumed $\F_{\infty-} = \F_\infty$.

\begin{definition}
Let $k<\infty$, $\tau\leq\infty$. Let $\{g_i\}$ be a family of measurable functions $g_i:\Omega\times(\bR^+\times E)^i\to \bR$. Then we define the $k$-fold iterated integral operator via the recursion
\[\I^{k}_\tau(\{g_i\}) = g_0 + \int_0^{T_{k}\curlywedge \tau} \I^{k-1}_t(\{g_{i-1}(t,x, \cdots)\}_{i=1}^{k}) dm,\]
with initial value $\I^0_\tau(g_0)=g_0$. For simplicity, we write $\I^k(\{g_i\}):=\I^k_\infty(\{g_i\})$.
\end{definition}

With this definition, the first few terms of our integral operator are
\[\begin{split}
   \I^\emptyset_\tau(\{g_0\}) &= g_0\\
   \I^{1}_\tau(\{g_0, g_1\}) &= g_0+\int_0^{T_1\curlywedge \tau} g_1(t,x)m(dt, dx)\\
   \I^{2}_\tau(\{g_i\}_{i=0}^2)& =g_0+\int_0^{T_2\curlywedge \tau}\left(g_1(t_1,x_1)+ \int_0^{T_1\curlywedge t_1} g_2(t_1,x_1,t_2,x_2) dm_2\right) dm_1\\
   \I^{3}_\tau(\{g_i\}_{i=0}^3)& =g_0+\int_0^{T_3\curlywedge \tau}\left(g_1+ \int_0^{T_2\curlywedge t_1} \left(g_2 +\int_0^{T_1\curlywedge t_2} g_3\, dm_3\right)dm_2\right) dm_1\\
  \end{split}
\]
The important point to notice is that the `internal' integrals are taken only up to the preceding jump times in our sequence.

We can now state a precursor to the chaos representation theorem, using the iterated integrals $\I$. 

\begin{theorem}\label{thm:chaosfiniteprecursor}
Let $Y\in L^2(\F_{T_k\curlywedge \tau})$ for $k<\infty$, and deterministic $\tau\leq\infty$. Then there exists a sequence of deterministic functions $\{g_i\}_{i=1}^k$ such that
\[Y=\I^{k}_\tau(\{g_i\}).\]
\end{theorem}

\begin{proof}
 First assume $\tau<\infty$. We shall use induction, iterating in $\alpha\leq k$ over the cases where $Y\in L^2(\F_{T_\alpha\curlywedge \tau})$. For the initial case, suppose $Y\in L^2(\F_0)$. Then $Y$ is a constant, so $Y= \I^0(g^0) = g^0$ for some constant $g^0$.

Suppose $Y\in L^2(\F_{T_\alpha\curlywedge \tau})$ and that the result holds for all $Y'\in L^2(\F_{T_{\alpha-1}\curlywedge t})$ for $t\leq \tau$. By the martingale representation theorem, $Y$ has a representation of the form
\[Y = E[Y] + \int_0^{T_\alpha\curlywedge \tau} \tilde g(\omega, t, x) dm\]
for some predictable function $\tilde g$ with $\|\tilde g\|_m<\infty$. As $\tilde g$ is predictable, for every $(t,x)$ the random variable $\tilde g(\cdot,t,x)$ is $(\F_{T_{\alpha-1}\curlywedge t})$-measurable. As we have
\[\begin{split}
 \infty&> \|\tilde g\|_m^2= \bE\left[\left(\int_0^{T_\alpha\curlywedge \tau} \tilde g dm\right)^2\right]\\
&=\sum_{\beta\leq \alpha}\int_0^{\tau-} \bE[I_{t\in ]T_{\beta-1}, T_\beta]} \tilde g^2(t,x)] d\zeta^\beta
\end{split}
\]
we know $\bE[I_{t\in ]T_{\beta-1}, T_\beta]} \tilde g(t,x)^2]<\infty$ $\zeta^\beta$-a.s. for all $\beta\leq \alpha$. Taking the sum over $\beta\leq \alpha$, we see $\bE[\tilde g(t,x)^2]<\infty$. Therefore $\tilde g(t,x) \in L^2(\F_{T_{\alpha-1}\curlywedge t})$. As we have supposed that the result holds on $L^2(\F_{T_{\alpha-1}\curlywedge t})$, we can find deterministic $\{g^{(t,x)}_i(\cdots)\}_{i=0}^{\alpha-1}$ such that
\[\tilde g_n(\omega,t,x)= \I^{\alpha-1}_t(\{g_i^{(t,x)}(\cdots)\})\]
from which we define
\[g_i(t,x, \cdots) = g^{(t,x)}_{i-1}(\cdots) \text{ for }i>1; \quad g_0 = E[Y].\]
This yields the representation of $Y$,
\[Y = \I^{\alpha}_\tau(\{g_i\}) = E[Y] + \int_0^{T_\alpha\curlywedge \tau} \I^{\alpha-1}_t(\{g_i^{(t,x)}\}) dm.\]

By induction, the result is proven for $Y\in L^2(\F_{T_{k}\curlywedge \tau})$ for all $k<\infty$. We now seek to let $\tau\to\infty$. This is easily done by the convergence of square-integrable martingales. For $Y\in L^2(\F_{T_k})$, let $Y_\tau :=\bE[Y|\F_{\tau}]$, so that $Y_{\tau-}\in L^2(\F_{T_k\curlywedge \tau})$. Therefore $Y_{\tau-} = \I^k_\tau(\{g_i^\tau\})$ for some collection of functions $\{g_i^\tau\}$. It is easy to verify that these functions are consistent, that is, $g_i^\tau = g_i^{\tau'}$ on $[0, \tau\wedge \tau'[\times E$, and hence $\{g_i^\tau\}$ can be taken to be independent of $\tau$. Therefore, by martingale convergence,
\[Y\leftarrow \bE[Y|\F_{\tau-}]=Y_{\tau-} = \I^k_\tau(\{g_i\}) \to \I^k_\infty(\{g_i\})\quad a.s.\]
from which we see $Y=\I^k_\infty(\{g_i\})$.
\end{proof}

\begin{remark}
 Intuitively, this representation in terms of $\I^k$ has a simple interpretation. From the martingale representation theorem, we know we can write any $\F_{T_k\curlywedge \tau}=\F_{T_k}\cap\F_{\tau-}$-measurable random variable in terms of the stochastic integral on $[0,T_k]\cap[0,\tau[$ of a predictable process $\tilde g_t$. However, up to time $T_k$, a predictable process $\tilde g_t$ is $\F_{T_{k-1}\curlywedge t}$-measurable for each $t$, and so by induction can itself be written as an integral on $[0,T_{k-1}]\cap [0,t[$. Hence any $T_k$-measurable random variable can be written as the iterated stochastic integral, where each integral is at most up to an earlier jump time.
\end{remark}

We can now construct the chaos representation.

\begin{definition}
 For $T\leq\infty$ a stopping time, we shall write
\[\bS^n_T:= \{(s_1, s_2,...,s_n): 0\leq s_n<s_{n-1}<...<s_1\leq T\}\subset [0,T]^{n}\]
For $T$ a stopping time, we define the $n$-fold iterated integral
\[\begin{split}
   J^n_T(g) &= \int_{\bS^n_T} g(\{(s_k, x_k)\}) \bigotimes^n_{k=1} m(ds_k,dx_k)\\
&=\int_0^T\int_0^{s_1-}\int_0^{s_{2}-}...\int_0^{s_n-} g(...) dm_{n-1}\,...\, dm_2 \, dm_1.
  \end{split}
\]
For convenience, $J^0_T(g) := g$ for all constants $g$.
\end{definition}

\begin{definition}
For $T$ a stopping time, let \[H_T^m:=\overline{\mathrm{span}}\{J^n_T(g): \bE[(J^n_T(g))^2]<\infty, n\leq m\},\] the $L^2(\bP)$-closure of the span of the square integrable iterated stochastic integrals of order at most $m$ up to $T$. This is a Hilbert space, with the same inner product as $L^2(\bP)$.
\end{definition}

We can now prove the the Chaos representation theorem for random variables known after finitely many jumps.
\begin{theorem}\label{thm:chaosfinitejump}
 For any $k<\infty$,  we have $H_{T_k}^k = L^2(\F_{T_k})$.
\end{theorem}
\begin{proof}
 Clearly $H_{T_k}^k \subseteq L^2(\F_{T_k})$, and $H_{T_k}$ is a Hilbert subspace. Therefore, either $H_{T_k}^k=L^2(\F_{T_k})$ or there exists a nonzero random variable $Y\in L^2(\F_{T_k})$ which is orthogonal to every element of $H_{T_k}^k$.

By Theorem \ref{thm:chaosfiniteprecursor}, the space spanned by the iterated integrals $\I^k$ is $L^2(\F_{T_k})$. Hence $Y$ has a representation of the form $Y=\I^k(\{g_i\})$, for some functions $\{g_i\}$. We seek to show that $g_i=0$ for all $i$ on the relevant range of integration. We shall do this using induction, however due to notational complexity, we shall simply write out the first three steps, the rest follow in the same manner.

\textbf{For $g_0$,} note that we must have
\[\begin{split}
   0&= \bE[Y J_{T_k}^0(g_0)] \\
&= \bE\left[\left(g_0+ \int_0^{T_k} \I^{k-1}_t (\{g_{i-1}(t,x,\cdots)\})dm\right) (g_0)\right]\\
&=g_0^2 + g_0 \bE\left[\int_0^{T_k} \I^{k-1}_t (\{g_{i-1}(t,x,\cdots)\})dm\right]\\
&=g_0^2
  \end{split}
\]
and we see $g_0\equiv 0$.

\textbf{For $g_1$,} note that as $g_0\equiv 0$, we know
\[Y = \int_0^{T_k}\I^{k-1}_{t_1}(\{g_{i-1}(\cdots)\}) dm_1 = \int_0^{T_k}(g_1(t_1,x_1) + \xi(t_1,x_1)) dm_1 .\]
 where $\xi(t_1,x_1) = \int_0^{T_{k-1}\curlywedge t_1} \I^{k-2}_{t_2}(\{g_{i-2}(t_1,x_1,\cdots)\})dm_2$. Note that $\bE[\xi(t,x)] \equiv 0$. Then
\[\begin{split}
0&= \bE[Y J^1_{T_k}(g_1)]\\
&=\bE\left[\left(\int_0^{T_k}(g_1(t_1,x_1) + \xi(t_1,x_1)) dm_1 \right)\left(\int_0^{T_1} g_1 dm_1\right)\right]\\
&=\|I_{t_1\leq T_k}g_1(t_1,x_1)\|_m^2 + \sum_\alpha \int_0^\tau \bE\left[\xi(t_1,x_1))\right] g_1(t_1, x_1) d\zeta^\alpha\\
&=\|I_{t_1\leq T_k}g_1(t_1,x_1)\|_m^2
  \end{split}
\]
and so $g_1(t_1,x_1)\equiv 0$ on $[0,T_k]$ (up to a set of $\zeta^\alpha$-measure zero for all relevant $\alpha\leq k$).

\textbf{For $g_2$,} note that as $g_0=g_1=0$, we know
\[\begin{split}
Y&= \int_0^{T_k}\int_0^{T_{k-1}\curlywedge t_1}\I^{k-2}_{t_2}(\{g_{i-2}(\cdots)\}) dm_2 dm_1 \\
&= \int_0^{T_k}\int_0^{T_{k-1}\curlywedge t_1}(g_2(t_1,x_1,t_2, x_2) + \xi(t_1,x_1,t_2, x_2)) dm_2 dm_1.
\end{split}\]
where $\bE[\xi(...)]\equiv 0$. Hence, expanding in the same way as above
\[\begin{split}
0&= \bE[Y J^2_{T_k}(g_2)]\\
&=\sum_{\alpha\leq k}\int_0^\infty \bE\left[I_{t_1\in ]T_{\alpha-1}, T_\alpha]} \left(\int_0^{T_{k-1}\curlywedge t_1} g_2(t_1,x_1 t_2,x_2)dm_2\right)^2\right]d\zeta^\alpha_1\\
&=\sum_{\alpha\leq k}\sum_{\beta\leq \alpha}\int_0^\infty\int_0^{t_1} \bE[I_{t_1\in ]T_{\alpha-1}, T_\alpha]}I_{t_2\in ]T_{\beta-1}, T_\beta]}] (g_2(t_1, x_1, t_2,x_2))^2 d\zeta^\beta_2 d\zeta^\alpha_1\\
  \end{split}
\]
and so $g_2\equiv 0$ up to a set of measure zero, on its relevant domain.

\textbf{Continuing the induction}, we see that $g_i \equiv 0$ for all $i\leq k$. Therefore $Y\equiv 0$, and there is no element of $L^2(\F_{T_k})$ orthogonal to all of $H_{T_k}$. Therefore the spaces coincide.
\end{proof}

Finally, we can expand our Chaos representation theorem to all of $L^2(\F)$.

\begin{theorem}
 Any square integrable random variable can be arbitrarily well approximated in $L^2$ by a sum of iterated integrals, or equivalently,
\[L^2(\F) = \overline{\left\{\cup_k H^k_{T_k}\right\}}.\]
\end{theorem}
\begin{proof}
 We shall again use the convergence of square integrable martingales. For any $Y\in L^2(\F)$, let $Y_k = \bE[Y|\F_{T_k}]$. This is a martingale in the discrete filtration $\G_k=\F_{T_k}$, and so $Y_k \to Y$ in $L^2$. Hence for any $\epsilon>0$ there exists a $k$ such that $\bE[(Y-Y_k)^2]<\epsilon/4$. By Theorem \ref{thm:chaosfinitejump} there is also a sequence $\{g_n\}_{n=1}^k$ such that $\bE\left[\left(Y_k - \sum_{n=0}^k J^n_{T_k}(g_n)\right)^2\right]<\epsilon/4$. By the triangle inequality,
\[\bE\left[\left(Y- \sum_{i=1}^k J^n_{T_k}(g_n)\right)^2\right]<\epsilon.\]
\end{proof}

\section{Conclusions}
We have shown general conditions such that an  arbitrary marked point process generates a martingale random measure $m$ for which a martingale representation theorem holds, and such that $L^2(\F)$ admits a chaos decomposition. 

A key motivating application of this result is to allow a general development of Malliavin calculus for marked point processes. This development is done in \cite{Decreusefond1998}, under the assumption that a chaos decomposition exists.

The only assumptions that we have made on the processes in question is that the compensating measure $\tilde p^\alpha$ is equivalent to some deterministic measure $\zeta^\alpha$, which is continuous in time and finite for finite times. It seems reasonable that some relaxation of this assumption is possible, for example, to only assuming that $\tilde p^\alpha$ is absolutely continuous with respect to $\zeta^\alpha$. The difficulty in doing this arises as we cannot then write the quadratic variation of a martingale in terms of a product measure $\bP\times \zeta$, and therefore cannot directly show that iterated integrals of different orders are orthogonal.

It also seems reasonable that a relaxation of the assumption that there are finitely many jumps should be possible. In \cite{Elliott1976} no such assumption is made, however this leads to the need for transfinite induction in the proof of the martingale representation theorem. Having convergent sequences of jumps also requires a relaxation of the continuity of $\tilde p$ in time (which we assume through the equivalence of $\tilde p^\alpha$ and $\zeta^\alpha$, coupled with the continuity of $\zeta^\alpha$). This may also be possible, however it leads to a more complex quadratic variation for stochastic integrals, as it allows the possibility of accessible jump times.

\bibliographystyle{plain}
\bibliography{../RiskPapers/General}
\end{document}